\documentclass[12pt]{article}
\usepackage[dvipsnames]{xcolor}
\usepackage{amsfonts,mathtools,amsthm,hyperref,babel,xcolor}
 \usepackage{mathabx} 
\usepackage[T2A]{fontenc}
\usepackage{float}

\newtheorem{thm}{Theorem}[section]
\newtheorem{lemma}[thm]{Lemma}
\newtheorem{prop}[thm]{Proposition}
\newtheorem{conj}[thm]{Conjecture}
\newtheorem{remark}[thm]{Remark}
\def\P{\mathbb{P}}
\def\R{\mathbb{R}}
\def\Z{\mathbb{Z}}
\def\E{\mathbb{E}}
\def\F{\mathcal{F}}
\def\G{\mathcal{G}}
\def\d{\delta}

\numberwithin{equation}{section}

\begin{document}

\title{Forest Fire Model on $\Z_{+}$ with Delays}
\author{Satyaki Bhattacharya, Stanislav Volkov}
\maketitle
\begin{abstract}We consider a generalization of the forest fire model on $\Z_+$ with ignition at zero only, studied in~\cite{Stas}. Unlike that model, we allow delays in the spread of the fires and the non-zero burning time of individual ``trees''. We obtain some general properties for this model, which cover, among others, the phenomena of an ``infinite fire'', not present in the original model. 
\end{abstract}

\noindent
{\bf Mathematics Subject Classification}: primary 60G55, 60K35; secondary 60F20

\noindent
{\bf Keywords and phrases}: forest fire model

\section{Introduction}
The following forest fire model was introduced in \cite{BT,Stas}; an extension of this model was also studied in~\cite{Comets}. Each node of $\Z_+$ can be vacant or occupied by a tree. If a node is vacant, and only in this case, a single tree appears at this node after an exponentially distributed time with a rate $1$; this time is independent of everything else. A constant source of fire is located at node $0$, and whenever a tree catches fire, the fire spreads instantly to the next tree on the right and continues doing so until the whole cluster is burnt, which takes zero time. The fire stops only when it reaches a vacant node (i.e.\ without a tree). 

Various forest fire models have been abundant, especially in the physics literature, partly for their practical applications and partly for the fact that they exhibit so-called {\em self-organized criticality}, which is a property of many dynamical systems that naturally evolve into a critical state, and where small changes can, surprisingly, trigger large-scale events. This phenomenon can be useful for explaining why complex systems (e.g., earthquakes, avalanches, stock markets, etc.) exhibit power-law behaviour without the necessity of external fine-tuning. Self-organized criticality is considered as a mechanism by which complexity arises in nature; for a more detailed analysis of this notion, we refer the reader to~\cite{Bak}. For examples of forest-fire models in the physics literature see e.g.~\cite{Drossel1}, where the authors compute the average number of trees destroyed by lightning, derive scaling laws and calculate all critical exponents; in~\cite{Drossel2}  the analytic solution of the self-organized critical forest-fire model in one dimension is given.  There are some interesting relations between percolation theory and forest fire models: in~\cite{vdB}, the authors study the connection between frozen percolation and forest fire processes without recovery; in~\cite{Ahlberg} it is shown that in higher dimensions a similar model introduced in~\cite{BB1} will eventually recover from fires. An extensive and detailed review of the forest models, including some practical ones, can be found in~\cite{Bressaud}. For other and more recent relevant models, please see~\cite{Ahlberg,BJ,BB,Bressaud,Comets,CFT,Kiss} and references therein.

In the current paper we consider a generalization of the model from~\cite{Stas} by allowing non-zero tree burning times, as well as by introducing possible delays, which affect how the fire spreads from a burning site to its right neighbour. Neither of these possibilities was considered in~\cite{Comets} or~\cite{Stas}, which motivated us to write the current paper. In our model, each site can be in three possible states: (a) vacant (no tree); (b) having a (non-burning) tree; (c) having a burning tree. Formally, assume that all sites are vacant initially and
\begin{itemize}
\item[(1)] the trees appear at each vacant, non-burning site $x\in\Z_+$ at a rate $1$ independently of anything, after which they can be burned by a fire arriving from the left, which will eventually return the site to the original {\em vacant} state;

\item[(2)] there is a constant source of fire attached to site $x=0$, so whenever a tree appears there, it starts burning immediately;

\item[(3)] when a tree at position $x$ catches fire for the $i$-th time ($x,i\in\Z_+$), denote this time as $f_{x,i}$, it takes $\theta_{x,i}\ge 0$ time to burn, where $\theta_{x,i}$  are drawn independently from some non-negative distribution $\theta$, so the state $x$ is ``burning" during the time interval $[f_{x,i},f_{x,i}+\theta_{x,i}]$;

\item[(4)] when a tree at position $x$ burns for the $i$-th time, the fire may {\bf spread} to tree $(x+1)$ in time $\Delta_{x,i}\ge 0$, where $\Delta_{x,i}$ are drawn independently from some non-negative distribution~$\Delta$, and the counting starts only after both the fire starts at $x$ and the tree at $(x+1)$ appears. Therefore, if the fire starts at position $x$ at time $t:=f_{x,i}$, then the fire at position $(x+1)$ starts at time $f_{x+1,j}=t+\Delta_{x,i}$ if there is already a tree\footnote{which is {\em not burning} at this time} at $(x+1)$ by time $t+\Delta_{x,i}$. On the other hand, if the site $x+1$ is vacant at time $t+\Delta_{x,i}$, assuming that the tree at $x+1$ appears at time $t+\Delta_{x,i}+\nu$ for some $\nu>0$, it will either start burning immediately at time $f_{x+1,j}=t+\Delta_{x,i}+\nu$ if $\nu<\theta_{x,i}$, or if $\nu\ge \theta_{x,i}$ the fire will not propagate beyond site $x$; we include the possibility that the fire at position $x$ goes off before the position $x+1$ catches fire\footnote{in real life, this can be interpreted as grassland between trees spreading fire, or that a burning tree emits sparks which are spread by the wind until they reach the next tree. Thus, the next tree might start burning after the tree on its left had already burnt};
   
\item[(5)] while a tree is burning, another tree cannot grow at the same site, nor can the fire from the left neighbour of this site spread to it. So the next tree at that spot can appear only after it was burnt completely and became vacant again, and it will take again $\exp(1)$ time independently of anything.
\end{itemize}

Observe the following:
\begin{itemize}
\item we may have more than one fire running in the model at the same time (but in different locations);
\item if the $\Delta$s are truly random, it is conceivable that while site $x$ has consecutive fires at times~$f_{x,i}$ and~$f_{x,i+1}$, respectively (such that $f_{x,i+1}>f_{x,i}+\theta_{x,i}$),  the time $\tilde t=f_{x,i}+\theta_{x,i}+\Delta_{x,i}$ by which the $i$-th fire at $x$ {\em stops affecting} site $x+1$ exceed the time $f_{x,i+1}+\Delta_{x,i+1}$ by which the next $(i+1)$-st fire at $x$ {\em starts affecting} $x+1$. In this case, we postulate that the  $(i+1)$-st fire at $x$ can ignite site $x+1$ {\em only after} the time $\tilde t$;
\item it is conceivable that a fire never stops but continues to infinity; we will call this phenomenon {\em an infinite fire}.
\end{itemize}

Now we introduce some notations. Let $\eta_x(t)\in\{0, 1\}$ be the state of the site $x \in \Z_+$ at time $t\ge 0$, and we say that the site $x$ is vacant (occupied, respectively) if $\eta_x = 0$ ($\eta_x = 1$ respectively). A vacant site $x\in\Z_+$ becomes occupied at rate $1$; once $x$ is occupied, it becomes vacant again after it catches fire that spreads from the site $x-1$ and stops burning after a time distributed as $\theta$. We assume that the process $\eta_x(t)$ is right-continuous.

Let $\nu_k$ be the time at which the $k$-th tree appears at the origin (and instantaneously begins burning); if $\theta\equiv0$, the times $\nu_1,\nu_2,\dots$ form a Poisson point process with rate $1$, otherwise
\begin{align}\label{eq:nus}
\nu_k=(\xi_1+\theta_{0,1})+(\xi_2+\theta_{0,2})+\dots+(\xi_{k-1}+\theta_{0,k-1})+\xi_k    
\end{align}
where $\xi_k$ are i.i.d.\ $\exp(1)$ random variables.

Assume for the moment that $\theta\equiv 0$.
Let $n_k$ denote the right-most point burnt by fire number~$k$; this fire will eventually spread to the site $n_k$  which will happen at time $\nu_k+\zeta_k$  where $\zeta_k=\Delta_0'+\Delta_1'+\dots+\Delta_{n_k-1}'$ with $\Delta'$s being i.i.d.\ copies of $\Delta$. We can interpret $\zeta_k$ as the duration of the $k$-th fire. Then
\begin{align*}  
\eta_1(\nu_k+\Delta_0'-)=
\eta_2(\nu_k+\Delta_0'+\Delta_1'-)=
\dots
=\eta_{n_k}(\nu_k+\Delta_0'+\Delta_1'+\dots+\Delta_{n_k-1}'-)&=1;\\
\text{but}\quad \eta_{n_k+1}(\nu_k+\Delta_0'+\Delta_1'+\dots+\Delta_{n_k}'-)&=0;
\end{align*}

If $n_k=\infty$ and $\zeta_k=\infty$, we shall call the $k$-th fire {\em an infinite fire}. Finally, we can similarly define $n_k$ and the duration of the $k$-th fire in the case when $\theta$ is not identically equal to zero, but we will not get such simple formulae.

\subsection{Graphic representation}
We can represent our model in the upper right quadrant of $\R^2$ as follows. For each $x\in\{0,1,2,\dots\}$ consider a vertical ray $(x,t)$, $t\ge 0$.  Suppose a tree appears at site $x$ for the $i$-th time at time $t$, and starts burning at time $t'\ge t$. Then colour the segment with the vertical coordinates $[t,t']$ \textcolor{Green}{green}. The tree will burn between times $t'$ and $t'+\theta_{x,i}$ so colour the interval $(t',t'+\theta_{x,i})$ \textcolor{Red}{red}. The rest of the line remains uncoloured. As a result, consecutive segments on this ray alternate: the uncoloured segments are followed by green and then by red segments (it is conceivable that a green segment degenerates to a single point if the fire at $x$ starts immediately after the tree is planted there; this can happen if $\theta_{x-1,j}>0$. Please see Figure~\ref{fig1} for the case when the duration of fire is always zero, and 
Figure~\ref{fig2} for the general case. Interestingly, the graphic representation of the process is somewhat reminiscent of direct percolation models, though we could not find a direct connection which would be useful in obtaining additional insights.

\begin{figure}[!htb]
    \centering
    \begin{minipage}{.45\textwidth}
        \centering
        \includegraphics[width=0.8\textwidth]{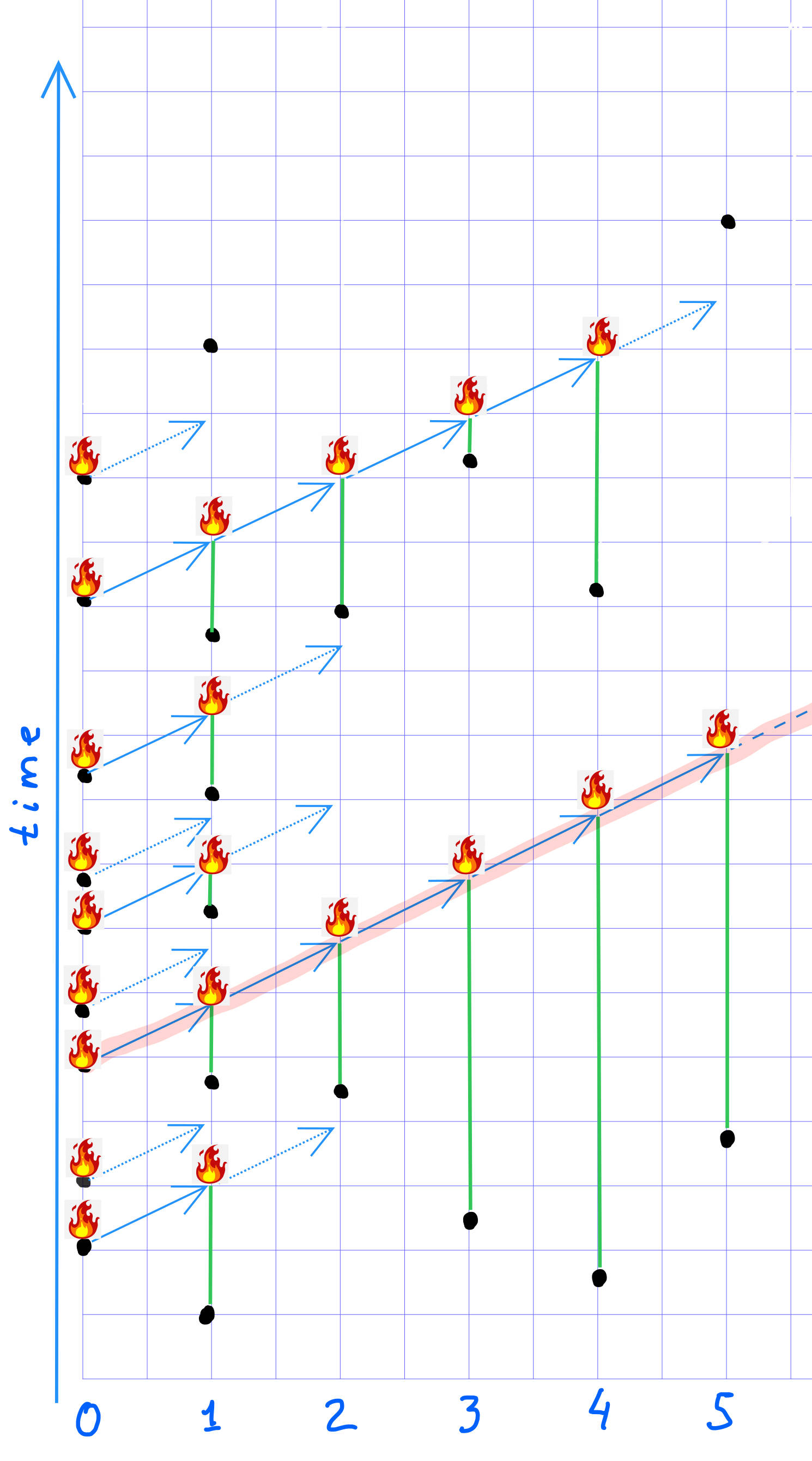}
        \caption{Graphical representation in the case $\theta\equiv 0$, and a constant $\Delta\equiv a>0$. Solid arrows represent the spread of fire, while dotted arrows indicate unsuccessful attempts to spread. The black dots mark the appearance of trees, and green lines denote periods during which a tree occupies a site before it burns. The highlighted line of arrows represents the infinite fire.}
        \label{fig1}
    \end{minipage}%
    \qquad
    \begin{minipage}{0.45\textwidth}
        \centering
        \includegraphics[width=0.83\textwidth]{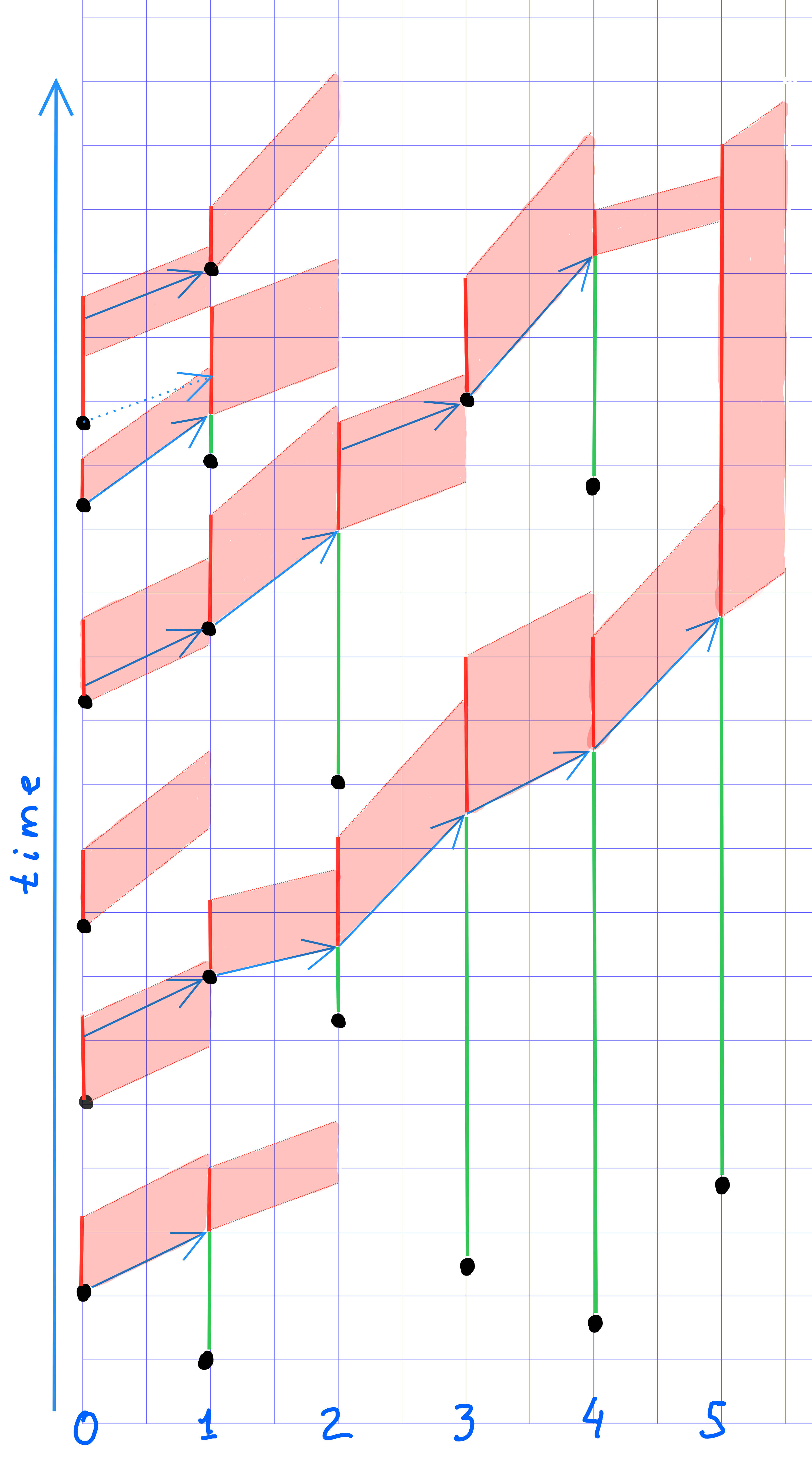}
        \caption{
        The black dots correspond to the appearance of trees; the green (red resp.) segments are the periods when a site is occupied by a ``healthy'' (burning resp.) tree. The shaded areas represent the periods when a burning tree at site $x$ affects site $x+1$. The solid arrows show the fire spreading to a neighbouring tree; the dotted arrows represent the situations when the fire {\em tried} to spread unsuccessfully.}
        \label{fig2}
    \end{minipage}
\end{figure}

If the $k$-th fire does not stop at $x$, and $x$ started burning at time $t$, then we can draw an arrow from some point in the red segment with coordinates $(x,s)$ with $s\in[t,t+\theta_{x,i})$ to the point $(x+1,s+\Delta_{x,i})$ such that the latter point is green, and each fire can be represented by a sequence of arrows with coordinates
\begin{align*}
(0,\nu_k)&\to(1,r_{1,k})\\    
(1,r_{1,k})&\to(2,r_{2,k})\\    
\dots\\
(n_{k-1},r_{n_{k-1},k})&\to(n_{k},r_{n_{k},k})
\end{align*}
such that the slope of the $x$-th arrow is $\Delta_{x,k}$ and the arrow always points to a green vertex.

\begin{remark}
As it follows from the description of the model, distinct arrows from $x$ to $x+1$ do not cross.   
\end{remark}

\subsection{Structure of the paper}
The rest of the paper is organized as follows. In Section~\ref{secEIF} we prove the existence of the phenomenon called ``the infinite fire'' in case of non-trivial but identically distributed spreading times~$\Delta$. In Section~\ref{secIB} we consider the situation when the burning time $\theta$ is identically zero, yet the spreading time is either constant or asymptotically decreasing, and establish conditions for (non)-existence of an infinite fire. In Section~\ref{secIS}, on the other hand, we show that if the spreading of the fire is immediate ($\Delta=0$) then no infinite fire occurs, and we also obtain some results on how quickly the fire spreads.

\section{Existence of infinite fire}\label{secEIF}
The infinite fire is a phenomenon which cannot happen in the forest fire models studied in either~\cite{Comets} or~\cite{Stas}; it arises, however, when there is a delay in the spreading of fire.
\begin{thm}\label{t:infirexists}
Suppose that $\P(\Delta>0)>0$. Then a.s.\ there will be an infinite fire, i.e. $\P(\inf\{k:\ n_k=\infty\}<\infty)=1$.
\end{thm}

The next auxiliary statement is quite obvious, yet, for the sake of completeness, we present its short proof.
\begin{lemma}\label{lem:lem}
Suppose that $\Delta$ is a non-negative random variable with $\P(\Delta>0)>0$. Then for some constants $\d_1>0$ and $\d_2>0$
$$
\P\left(\sum_{j=1}^n \Delta_j'\ge n\d_1\text{ \ for all }n\ge 1\right)\ge \d_2
$$ 
where $\Delta_j'$ are i.i.d.\ copies of $\Delta$.
\end{lemma}
\begin{proof}
By the condition of the lemma, there exist  $a_1>0$ and $p_1>0$ such that $\P(\Delta\ge a_1)\ge p_1$.
Hence $\Delta_j'$ is stochastically larger than
$$
Z_j=\begin{cases}
a_1, &\text{with probability }    p_1;\\
0, &\text{with probability }    1-p_1
\end{cases}
$$
and we can assume that $Z_j$s are also i.i.d.
Consequently, it suffices to show that
$$
\P\left(\sum_{j=1}^n Z_j\ge n\d_1\text{ for all }n=1,2,\dots\right)\ge \d_2
$$ 
where $\delta_1:=\E Z_j/2=a_1p_1/2>0$. 

By the strong law of large numbers $\frac1n\sum_{j=1}^n Z_j\to a_1p_1$ a.s.\ as $n\to\infty$, hence, since $\d_1<\E Z_J$, a.s.\ there exists a (random)~$N=N(\omega)$ such that $\sum_{j=1}^n Z_j\ge n\d_1$ for all $n \ge N$. Since $N$ is finite, for some non-random $n_0\ge 1$ we have $\P(N\le n_0)\ge 1/2$. As $Z_j\in\{0,a_1\}$, we have
\begin{align*}
\frac12&\le \P\left( \sum_{j=1}^n Z_j\ge n\d_1\ \forall n> n_0\right)
\le
\P\left( \sum_{j=1}^n Z_j\ge n\d_1\ \forall n> n_0\mid Z_1=Z_2=\dots=Z_{n_0}=a_1\right)
\\ & =\frac{\P\left( \sum_{j=1}^n Z_j\ge n\d_1\ \forall n> n_0\text{ and } Z_1=Z_2=\dots=Z_{n_0}=a_1\right)}{\P\left( Z_1=Z_2=\dots=Z_{n_0}=a_1\right)}
\\ & =\frac{\P\left( \sum_{j=1}^n Z_j\ge n\d_1\ \forall n\ge 1\text{ and } Z_1=Z_2=\dots=Z_{n_0}=a_1\right)}{p_1^{n_0}}
\le \frac{\P\left( \sum_{j=1}^n Z_j\ge n\d_1\ \forall n\ge 1\right)}{p_1^{n_0}}
\end{align*}
so the Lemma holds with e.g.\ $\d_2:=p_1^{n_0}/2>0$.    
\end{proof}

\begin{proof}[Proof of Theorem~\ref{t:infirexists}]
First, we will show that for any $x\in \Z_+$ the fire will eventually reach~$x$. In fact, by induction we will show a stronger statement: for each $x=0,1,2,\dots$ there will be infinitely many fires that reach the site $x$ and they will occur at arbitrary large times (i.e., all~$f_{x,i}$s are  finite, and $f_{x,i}\nearrow\infty$ as~$i$ increases.) The proof is based on induction. 

For site $x=0$ this fact is trivial as $f_{0,k}\equiv \nu_k$, see~\eqref{eq:nus}. Now suppose that our statement is true for some $x\geq 0$. This means that there are infinitely many fires that start at times $f_{x,1}<f_{x,2}<\dots$. Since the probability of not having a single tree at $x+1$ by time $s$ is $e^{-s}\downarrow 0$ and $\lim_{i\to\infty}f_{x,i}=\infty$, eventually for some time $f_{x,j}$ there will be a tree at $x+1$ so it will start burning at time $f_{x,j}+\Delta_{x,j}<\infty$ and burn during time $\theta_{x+1,1}\ge 0$; thus we ensured that there will be at least one fire at $x+1$. Now consider the process after the time $f_{x,j}+\Delta_{x,j}+\theta_{x+1,1}$. It takes an exponentially distributed time for a new tree to appear at $x+1$; after that, it will eventually be burnt by a fire started at site $x$ at some time $f_{x,j'}$ with $j'>j$. This fire will start burning the tree at $x+1$ again at time $f_{x,j'}+\Delta_{x,j'}$. Similar arguments hold for the third fire, the fourth fire, etc. The exponential waiting time for the appearance of new trees also implies that $f_{x+1,j}$ is stochastically larger than the time of arrival of the $j$th point of a Poisson process with rate $1$, yielding $f_{x+1,i}\to\infty$. The induction step is finished.
\vskip 5mm

Let $m_1$ denote the location of the tree, farthest from the origin, burnt by the first fire; if the fire continues indefinitely, we let $m_1=+\infty$. For $k\ge 2$, we recursively define $m_k$ as the location of the tree, farthest from the origin, burnt by the first fire which reaches beyond $m_{k-1}$; we can think of $m_k$s as the successive maxima reached by the fires.  If $m_{k-1}=\infty$ for some $k$, we set $m_{k}=m_{k+1}=\dots=\infty$ as well. The successive maxima obviously satisfy  $0\le m_1<m_2<m_3<\dots$ (with the convention $\infty<\infty$).

Let $A_k=\{m_k<\infty\}$, i.e., the $k$-th fire was finite, and let $T_k=f_{m_k,1}+\theta_{m_k,1}$ be the time when the site $m_k$ stops burning for the first time. Let $\F_k$ be the $\sigma$-algebra\footnote{note that this is a sigma-algebra generated by the stopping time rather than a discrete time} generated by the states of sites $\{0,1,\dots,m_k+1\}$ during the period $[0,T_k]$; {\em we do not include in it} any events related to appearances of the trees at sites $m_k+2,m_k+3,\dots$.  We want to compute a lower bound on $\P(A_{k+1}^c\mid \F_{k})$, assuming that $A_k$ has occurred and $m_k=m$. Suppose that the fire has reached $m+1$ by some time $t$ (note that $t>T_k$). It takes $\Delta_{m+1,1}$ units of time for the fire to reach $m+2$. Assuming that there is a tree at $m+2$ by that time (i.e., $t+\Delta_{m+1,1}$), the fire will continue to $m+3$ if there is a tree at $m+3$ by time $t+\Delta_{m+1,1}+\Delta_{m+2,1}$, and so on. Thus the fire will continue {\em ad infinitum}, provided that
\begin{align*}
\text{there is a tree at site $(m+i)$\ by time }\ t
+\Delta_{m+1,1}+\Delta_{m+2,1}+\dots++\Delta_{m+i-1,1}
\\ \text{ for each }i=2,3,\dots   
\end{align*}
Since the probability that there is no tree at site $x$ by time $s\ge 0$ is $e^{-s}$ (assuming no previous fires at $x$), the probability of the above event given $t$ and $\{\Delta_{m+i,1}\}_{i=1}^\infty$ is
\begin{align}\label{eqstar}
\prod_{i=2}^\infty \left(1-e^{-t-\sum_{k=1}^{i-1}\Delta_{m+k,1}}\right)\ge
\prod_{i=2}^\infty \left(1-e^{-\sum_{k=1}^{i-1}\Delta_{m+k,1}}\right).
\end{align}
At the same time, by Lemma~\ref{lem:lem}
$$
\P\left(e^{-\sum_{j=1}^n \Delta_j'}\le e^{-n\d_1}\text{ for all }n=1,2,\dots\right)\ge \d_2
$$ 
hence if we denote
$$
\epsilon_0:=\prod_{n=1}^\infty \left(1-e^{-n\d_1}\right)>0
$$
then the probability that the RHS of~\eqref{eqstar} is larger than $\epsilon_0$ is no less than $\d_2$. As a result,
$$
\P(A_{k+1}^c\mid\F_k)\ge \epsilon_0\delta_2>0\quad\text{ for all }k\ge 1
$$
and hence by the conditional Borel-Cantelli lemma there will be an infinite fire a.s.
\end{proof}

\section{Instant burning}\label{secIB}
In this Section, we will consider the case when $\theta=0$, and $\E\Delta_x>0$ that is, the trees are burnt instantly but the expected time for the fire to spread from site $x$ to  $x+1$ is strictly positive.

\begin{thm}\label{t:no2infirexists}
Suppose $\theta\equiv 0$, $0<\E\Delta<\infty$ and~$\mathrm{Var}(\Delta)>0$. Then there is only one infinite fire.
\end{thm}
\begin{remark}
The non-random case $\Delta\equiv \mathrm{const}>0$ will be covered later in Theorem~\ref{t:ed}.
\end{remark}
\begin{proof}
The existence of infinite fire follows from Theorem~\ref{t:infirexists} since $\E\Delta>0$ implies $\P(\Delta>0)>0$, so we only need to show that there cannot be  {\em a second infinite fire}.

For each site $x\in\Z_+$ let $i_x$ be {\em the index of the fire for this site}, which turned out to be the first infinite fire (hence there will be an $x_1$  such that for all $x>x_1$ we have $i_x=1$).

Assume now that there is a second infinite fire. Since for any $t>0$ the number of distinct fires at each site $x$ by time $t$ is a.s.\ finite, there will be a (random) site $x_2$ such that each site with $x>x_2$ was not burnt by any {\em finite fire} between the times when it was burnt by the first and the second infinite fire. Hence, for each $x>x^*=\max(x_1,x_2)$, the first (the second resp.) fire to burn $x$ is, in fact, the first  (the second resp.) infinite fire. Since burning is immediate ($\theta\equiv 0$), for $x>x^*$
\begin{align*}
 f_{x+1,1}=f_{x,1}+\Delta_{x,1}\quad\text{and}\quad f_{x+1,2}=f_{x,2}+\Delta_{x,2}.
\end{align*}
Let us define
$
S_x=
\sum_{k=1}^x\xi_k,
$
where $\xi_k=\Delta_{k,2}-\Delta_{k,1}$ are i.i.d.\ mean zero random variables. Then $f_{x,i_x+1}-f_{x,i_x}=C(x^*)+S_x$ for $x\ge x^*$  where $C(x^*)$ is some quantity depending on~$x^*$ only. 
Indeed, for all sites at locations $x>x^*$ the first fire to burn $x$ will be {\em the first infinite fire}, and the second fire to burn $x$ will be the second infinite fire. Since $S_x$ is a one-dimension zero-mean random walk, it is recurrent by \cite[P8, Section I.2]{Spitzer} and hence $\P(\liminf_{x\to\infty} S_x=-\infty\mid x^*)=1$, which contradicts the fact that we must always have $f_{x,i_x+1}>f_{x,i_x}$.
with probability $0$.
\end{proof}

\subsection{Constant spreading time}
Recall that the burning time $\theta$ is assumed to be zero, and assume additionally that $\Delta$ is the same non-random positive constant at all locations; in this Section we want to study this special case. An interesting fact, observed by Edward Crane, is that after the infinite fire, the model here can be coupled with that of~\cite{Stas}. 
We have already shown that there exists an infinite fire in this case; now we are interested in finding out the distribution of {\em when exactly} the infinite fire starts, and what the index $\kappa$ of that fire can be.
\begin{thm}\label{t:ed}
Assume that $\theta\equiv 0$ and $\Delta\equiv a>0$. Suppose that the infinite fire starts at site $0$ at some time $T\ge  0$ (which is, of course, random). Then this fire reaches point $x\ge 1$ precisely at time $T+ax$, and if we denote $\hat \eta_x(t)=\eta_x(t+T+ax)$, $t\ge 0$, then the process $\hat\eta$ is exactly the one introduced in~\cite{Stas}, i.e., the one with $\theta=\Delta\equiv 0$.
\end{thm}
\begin{proof}
This fact follows immediately from the definition of the process $\hat\eta_x(t)$ and can be also seen from the graphical representation of the fire process (see Figure~\ref{fig1}), where all the arrows have the same slope $a$, and the infinite fire corresponds to a straight line going from $(0,T)$ towards infinity at the constant slope $a$.
\end{proof}

To get an estimate of when the infinite fire starts, observe the following. The times of consecutive fires at the origin form a Poisson process, and thus the $k$th fire starts at time $\nu_k=\xi_1+\xi_2+\dots+\xi_k$ where $\xi_i$ are i.i.d.\ exponential(1) random variables (see~\eqref{eq:nus}). Let $\kappa\ge 1$ be the index of the infinite fire, thus $T=\nu_\kappa$. In the statement below we estimate the probability that the first fire is infinite.
\begin{prop}
Assume the same as in Theorem~\ref{t:ed}. Then
$$
1-\frac{1}{2\mu}\le \P(\kappa=1)\le \mu(1-e^{-1/\mu})
$$
where $\mu:=e^a-1$. Hence, for large $a>0$ we have $\P(\kappa=1)=1-\frac1{2(e^a-1)}+o(1)$.
\end{prop}
\begin{proof}
The first fire burns $0$ at time $\xi_1$. For this fire to spread indefinitely, one needs to have at least one tree at position $x$ at time $\xi_1+ax$ for each $x\in\{1,2,\dots\}$; given $\xi_1$, this probability can be computed as $\prod_{x=1}^{\infty}\left(1-e^{-\xi_1-ax}\right)$.

Since for any sequence $y_1,y_2,\dots\in[0,1]$  we have
\begin{align}\label{eq:bound}
1-\sum_{k=1}^\infty y_k \le \prod_{k=1}^\infty
(1-y_k)\le \exp\left(-\sum_{k=1}^\infty y_k\right),
\end{align}
and
$$
\sum_{x=1}^\infty e^{-\xi_1-ax}=\frac{e^{-\xi_1}}
{e^a-1}.
$$
By taking the expectation over $\xi_1$, we get
\begin{align*}   
\P(\kappa=1)&=\E \prod_{x=1}^{\infty}\left(1-e^{-\xi_1-ax}\right)\le 
\int_0^\infty\exp\left(-\frac{e^{-u}}
{e^a-1}\right) e^{-u}\mathrm{d}u
\\ &
=
\int_0^1\exp\left(-\frac{v}
{e^a-1}\right) \mathrm{d}v
=(e^a-1)\left(1-e^{-\frac1{e^a-1}}\right).
\end{align*}
On the other hand, by \eqref{eq:bound}
\begin{align*}   
\P(\kappa=1)&\ge 
\int_0^\infty\left(1-\frac{e^{-u}}{e^a-1}\right) e^{-u}\mathrm{d}u
=1-\frac1{2(e^a-1)},
\end{align*}
which completes the proof.
\end{proof}

\begin{remark}
A much better approximation 
$$
\P(\kappa=1)\approx 1 - \frac{e^{-a}+e^{-2a}}2 - \frac{e^{-3a}}6
$$
can be obtained by using a tighter lower bound of $\prod_{k=1}^\infty (1-y_k)$. It gives the relative error of $0.01$ or less for $a\ge 1/2$, and at most $0.07$ for $1/2\ge a\ge 1/3$.
\end{remark}

Suppose the first fire did not become the infinite fire; let $m_1$ be the right-most point burnt by it. By the arguments similar to those in Theorem~\ref{t:ed}, we can couple our process after  time~$\nu_1$ with the simple fire process  $\hat\eta_x(t)$ from~\cite{Stas} such that 
$$
\hat\eta_x(t)=\eta_x(\nu_1+ax+t)\text{ for } x\in\{0,1,\dots,m_1,m_1+1\}. 
$$
Let us also observe that given $\nu_1$,
$$
\P(m_1=k)=e^{-\nu_1-a(k+1)}\prod_{x=1}^{k}\left(1-e^{-\nu_1-ax}\right),\quad k=1,2,\dots.
$$
Let $\delta_1=1$ and, assuming that the first fire was not infinite, let $\delta_2$ be {\it the index} of the first fire which extends beyond $m_1$, i.e., burns site $m_1+1$. Denote by $m_2$ ($\ge m_1+1$) the right-most point burnt by this fire. Then, given the time when it started, that is $\nu_{\delta_2}$,
$$
\P(m_2=m_1+k\mid \nu_{\delta_2})=e^{-\nu_{\delta_2}-a(m_1+k+1)}\prod_{x=m_1+2}^{m_1+k}\left(1-e^{-\nu_{\delta_2}-ax}\right),\quad k=1,2,3,\dots.
$$
Similarly, one can define pairs $(\delta_3,m_3)$, $(\delta_4,m_4)$, etc., until $(\delta_N,m_N)$ where $\delta_N=\kappa$ and $m_N=+\infty$.

Now notice that from coupling with the process in~\cite{Stas}, we have $\nu_{\delta_{k+1}}-\nu_{\delta_{k}}=\hat\tau_{m_k+1}$ where~$\hat\tau_z$ is the first time when the original forest fire process\footnote{that is, the one with $\Delta=\theta=0$} burns site $z\ge 1$. 
Hence $\nu_{\delta_k}=\psi_k(m)$
where 
\begin{align}\label{eq:psi}
\psi_k(\mu)=\sum_{\ell=0}^k\hat\tau_{\mu_\ell+1}    
\end{align}
where we set $\mu_0:=-1$ and thus $\hat\tau_0$ has the same $\exp(1)$ distribution as $\nu_1$. As a result, for any positive integer $n$ and a sequence $0\equiv \mu_0<\mu_1<\mu_2<\dots<\mu_{n-1}$,
\begin{align*}
 &\P(\kappa\ge n,m_1=\mu_1,m_2=\mu_2,\dots,m_{n-1}=\mu_{n-1})
 \\ & =
 \E\left[
 \prod_{i=1}^{n-1} 
 e^{-\nu_{\delta_i}-a(\mu_{i-1}+1)}
 \prod_{x=\mu_{i-1}+1}^{\mu_{i}} 
 \left(1-e^{-\nu_{\delta_i}-ax}\right)
 \right]
 \\ & =
 \E\left[
 \prod_{i=1}^{n-1} \left(
 e^{-(\psi_{i-1}(\mu)+a(\mu_{i-1}+1))}
 \prod_{x=\mu_{i-1}+1}^{\mu_{i}} 
 \left(1-e^{-(\psi_{i-1}(\mu)+ax)}\right)\right)
 \right]
 \\ & =
 \E\left[\left(
 \prod_{i=1}^{n-1} 
 e^{-(\psi_{i-1}(\mu)+a(\mu_{i-1}+1))}
 \right)
 \prod_{x=1}^{\mu_{n-1}} 
 \left( 1-e^{ -(ax+\psi_{\min\{k:\ x\le \mu_k\}}(\mu)) } \right)
 \right].
 \end{align*}
 Thus
\begin{align*}
\P(\kappa>n)=
&\E\left[
\sum_{\mu_1=1}^\infty\sum_{\mu_2=\mu_1+1}^\infty\dots
\sum_{\mu_{n-1}=\mu_{n-2}+1}^\infty
\left(
 \prod_{i=1}^{n-1}  e^{-(\psi_{i-1}(\mu)+a(\mu_{i-1}+1))}
 \right)
 \right.\\ & 
 \qquad \left.
 \times \prod_{x=1}^{\mu_{n-1}} 
 \left(
 1-e^{ -(ax+\psi_{\min\{k:\ x\le \mu_k\}}(\mu)) }
 \right)
 \right]
\end{align*}
which expresses the probability in terms of random variables $\hat\tau_x$ whose distribution can be obtained explicitly from Lemma~1.1 in~\cite{Stas}, and then by using~\eqref{eq:psi}.

\subsection{Non-identically distributed  $\Delta$'s}
In this section, we modify the model to allow $\Delta_x$, $x=0,1,2,\dots$, to have different distributions, which may depend on the location.
\begin{thm} 
Let $c\le 1$. Assume that either 
\begin{itemize}
    \item[(a)] $\Delta_x\le c/x$ \ for all large $x$, or
    \item[(b)] for some $C>0$ we have $|\Delta_x|\le C$ for all~$x$,  $\sum_x Var(\Delta_x)<\infty$, $\left|\sum_x \left[\E\Delta_x-\frac{c}x\right]\right|<\infty$.
\end{itemize}
Then the infinite fire does not occur a.s. i.e., $\P(n_k<\infty\text{ for all }k )=1$.
\\[2mm]
Conversely, if $c>1$, and either 
\begin{itemize}
    \item[(c)] $\Delta_x> c/x$ \ for all large $x$, or
    \item[(d)] $\E \Delta_x=\frac {c+o(1)}x$  and $\displaystyle \P\left(\lim_{x\to\infty}\, |\Delta_x|\sqrt{x}=0\right)=1$.
\end{itemize}
Then the infinite fire occurs a.s.
\end{thm}
\begin{proof} Let $c\le 1$. We will only consider case (b), as case (a) is similar (and easier).
Suppose the fire reaches some site $x\ge 1$  for the first time at time $t_x$. Conditioned on this time $t_x$, the probability that this fire will continue indefinitely is
$$
q_x:=\prod_{n=x+1}^\infty \left(1-e^{-(t_x+\Delta_{x}+\Delta_{x+1}+\dots+\Delta_{n-1})}\right)
=\prod_{n=x+1}^\infty \left(1-e^{-(t_x+S_{x,n})}\right)
$$
as it takes $S_{x,n}:=\Delta_{x}+\Delta_{x+1}+\dots+\Delta_{n-1}$ units of time for the fire to reach site $n$ and we need to ensure that already there is a tree at this location.  Let $\xi_x=\Delta_x-\frac{c}x$. Then $\sum_x \xi_x$ converges a.s.\ by Kolmogorov's three-series theorem. 
 We have
$$
S_{x,n}=\sum_{i=x}^{n-1}\left(\xi_i+\frac{c}i\right)
=c\log\frac{n}{x}+o_x(1)+\sum_{i=x}^{n-1}\xi_i=c\log\frac{n}{x}+o_x(1)
$$
where $o_x(1)\to 0$ a.s.\ as $x\to\infty$.
As a result, 
 $$
 e^{-t_x-S_{x,n}}=e^{-t_x-o_x(1)}\,\left(\frac{x}{n}\right)^c
 $$ 
 which is not summable in $n\ge 1$ as $c\le 1$, yielding $q_x=0$.

\vspace{5mm}
\textbf{Now assume  $c>1$.} We will consider only the harder case (d). Let $\Delta_x^{(i)}$ be the (random) spread time from site $x$ to site $x+1$,  for the fire number $i=1,2,\dots$ started at the origin.  Note that some of the $\Delta_x^{(i)}$s will be irrelevant for the understanding of the process: for example, if the first fire reaches only site $1$, and the second fire reaches site $3$, then $\Delta_1^{(1)},\Delta_2^{(1)}$ will be irrelevant for the behaviour of the process, but $\Delta_1^{(2)},\Delta_2^{(2)}$ {\em will be} relevant.

Let $\zeta_n\in\{1,2,\dots\}$ be the index of the fire that burns $n$ for the very first time. Pick a small $\delta>0$ such that~$c-\delta>1$,  and let
\begin{align*}
D_{n,i}&:=\left\{\sum_{x=0}^{n-1} \Delta_x^{(i)}<(c-\delta)\log n\right\},
\\
A_n&:=\{\text{the first fire which reaches the $n-$th tree does not spread to the $(n+1)^{st}$ tree}\}
\end{align*}
Note that in these notations $\left(D_{n,\zeta_n}\right)^c=\{$the fire which reaches the $n$-th tree for the first time takes more than $(c-\delta)\log n$ units of time from its start$\}$.

We have $0\le \Delta_x\le C_x$ a.s., where  $C_x\le \frac\d {\sqrt {2x}}$ for all large $x$, yielding
$$
\sum_{x=1}^{n-1} C_x^2\le \d^2\log n
$$
for large $n$. Since  also
$$
\sum_{x=1}^n \E\Delta_x =(c+o(1))\log(n)
$$
by Hoeffding's inequality,
$$
\P\left(D_{n,i}\right)\le \exp\left(-\frac{2\delta^2\log^2 n }{\d^2\log n}\right)=\frac1{n^2}.
$$

To prove the existence of an infinite fire, it suffices to show that $A_n$ happens only {\em finitely often}, almost surely.  We have
\begin{align}\label{eqAn}
\begin{split}
 \P(A_n)&=\P(A_n\cap D_{n,\zeta_n}^c)+\P(A_n\cap D_{n,\zeta_n})
\leq \P\left(A_n\mid D_{n,\zeta_n}^c\right)+\P\left(D_{n,\zeta_n}\right)
\\&
=\sum_{i=1}^\infty \left[\P\left(A_n\mid D_{n,i}^c,\zeta_n=i\right)
+\P\left(D_{n,i}\mid \zeta_n=i\right)\right]
\times \P(\zeta_n=i)\\&
\leq
\sum_{i=1}^\infty
\left[\P(\text{no tree at $(n+1)$ upto time }(c-\delta)\log n)+\P(D_{n,i}\mid \zeta_n=i)\right]
\times \P(\zeta_n=i)
\\&=\frac{1}{n^{c-\delta}}+
\sum_{i=1}^\infty
\P\left(D_{n,i}\mid \zeta_n=i\right)\times\P(\zeta_n=i) 
\leq\frac{1}{n^{c-\delta}}+\frac{1}{n^2}
\end{split}
\end{align}
which holds as long as we can show that
\begin{align}\label{eqFin}
\P\left(D_{n,i}\mid \zeta_n=i\right) \leq \P\left(D_{n,i}\right).
\end{align}
Let us show that this is indeed the case. Observe that $\{\zeta_n=i\}=X\cap Y$, where
\begin{align*}
X&:=\{\text{the first $i-1$ fires do not reach site }n\},\\
Y&:=\{\text{the $i$-th fire reaches site }n\}.
\end{align*}
Note also that $X$ and $D_{n,i}$ are independent since the distributions of the first $(i-1)$ fires are not influenced by the $\{\Delta_0^{(i)},\Delta_1^{(i)},\dots,\Delta_{n-1}^{(i)}\}$,  and $\P(Y)\geq \P\left(Y\mid D_{n,i}\right)$,  since the probability of the $i$-th fire reaching $n$ increases whenever any of $\Delta_x^{(i)}$ in the above set increases. As a result,
\begin{align*}
\P\left(X\cap Y\mid D_{n,i}\right)
&=\P\left(X\mid Y,D_{n,i}\right)\, \P\left(Y\mid D_{n,i}\right)
=\P(X\mid Y)\, \P\left(Y\mid D_{n,i}\right)
\\ &
\leq \P(X\mid Y)\, \P(Y)=\P(X\cap Y)    
\end{align*}
which together  with the Bayes formula
\begin{align*}
\P\left(D_{n,i}\mid \zeta_n=i\right)
=\frac{\P\left(\zeta_n=i\mid D_{n,i}\right)\times \P\left(D_{n,i}\right)}
{\P\left(\zeta_n=i\right)}
\end{align*}
yields \eqref{eqFin}.

Finally, by~\eqref{eqAn}, since $c-\delta>1$, the probabilities $\P(A_n)$  are summable over $n$, and by the Borel-Cantelli lemma, $A_n$ occurs finitely often, almost surely. As a result, there cannot be infinitely many successive maxima, implying that an infinite fire a.s.\ exists.
\end{proof}

\section{Instant spread, non-zero burning times.}\label{secIS} 
Throughout this Section we assume that $\Delta\equiv 0$ but $\theta\ge 0$.
\begin{thm}\label{t:tau0}
Suppose that $\Delta\equiv 0$. Then a.s.\ there are no infinite fires.
\end{thm}
\begin{proof}
Recall that $f_{x,i}$ denotes the time when the site $x\in\Z_+$ is burnt for the $i$th time. Then at time $t=f_{x,1}$ there may already be trees at sites $x+1,x+2,\dots$ and the fire will spread to those sites immediately (as $\Delta=0$); however, since the probability of not having a single tree during time $s$ is $e^{-t}$ and these events are independent for each site $y> x$, with probability one there will be a site $y\ge x+1$ such that there are no trees there; we let $y$ be the leftmost such site. This means that 
$$
f_{x,1}=f_{x+1,1}=\dots=f_{y-1,1}<f_{y,1}.
$$
For this fire still to reach $y$, there must be a tree ``planted" at $y$ during the time interval $[t,t+\theta_{y-1,1}]$. Conditionally on $\theta_{y-1,1}$, this probability equals $1-e^{-\theta_{y-1,1}}$, so every time the fire spreads immediately from some site $x$ to $y-1$, with probability $p_*:=\E e^{-\theta}>0$, independently of the past, it will not spread further. Since this probability does not depend on $x$ or $s$ either, it means that a.s.\ the fire will eventually stop.
\end{proof}

\vspace{1cm}
Now we want to study some further properties of the fire process. Recall that $m_k$ denotes the location of the tree, farthest from the origin, burnt by the $k$-th fire. 
 
\begin{lemma}\label{Lemma 2}
Assume $\Delta\equiv 0$ and fix some $\epsilon>0$. The time $f_{m_i,1}$ when the fire reaches $m_i$ for the first time is less than $(1+\epsilon)\log m_i$ for all but finitely many $i$, almost surely.
\end{lemma}
\begin{proof}
Suppose that the fire reaches site $n$ at time $t=f_{n,1}$. The probability that there is no tree at $n+1$ at this time is $e^{-t}$, so in case $t\ge (1+\epsilon)\log n$, the probability of not spreading the fire further is less than 
$$
e^{-(1+\epsilon)\log n}=\frac1{n^{1+\epsilon}}
$$
Hence $\P(A_n)\le \frac1{n^{1+\epsilon}}$ where
$$
A_n=\{n\text{ is a local maximum  and }f_{n,1}>(1+\epsilon)\log n)\}
$$
so by the Borel-Cantelli lemma, $A_n$ occurs only for finitely many $n$s a.s. From this, the statement follows.
\end{proof}

The next statement is similar to Proposition~2 in~\cite{Comets}.
\begin{lemma}\label{Lemma 3}
Assume $\Delta\equiv 0$ and fix some $\epsilon>0$.  Let
$$
F_n=F_n(\epsilon)=\{f_{n,1}\ge (1-\epsilon)\log n\}.
$$
Then 
$$
\P\left(F_n^c\right)<e^{-n^\epsilon}
$$  
and hence the event $F_n^c$ a.s.\ occurs only finitely often.
\end{lemma}
\begin{proof}
In order to reach $n$ by time $t$, there should be a tree at each location $x\in [0,n]$ by that time (or even earlier, because $\theta>0$), so
$$
\P(f_{n,1}<(1-\epsilon)\log n)\le \left(1-e^{(1-\epsilon)\log n}\right)^{n}=\left[\left(1-\frac1{n^{1-\epsilon}}\right)^{n^{1-\epsilon}}\right]^{n^\epsilon}
<e^{-n^\epsilon}.
$$  
The second statement follows from the Borel-Cantelli lemma.
\end{proof}

Lemmas~\ref{Lemma 2} and~\ref{Lemma 3} give a tight bound on $f_{m_i,1}$, since by Lemma~\ref{Lemma 3}, that $f_{m_i,1}<(1-\epsilon)\log m_i$, only finitely often almost surely.
\\[5mm]

Recall that $f_{x,2}$ is the time when the site $x$ is burnt for the second time. We believe that the following holds, even though we cannot prove it.
\begin{conj}\label{conj1}
There is a $c\ge 1$, such that a.s.\ there exists a (possibly random) $\kappa$ such 
\begin{align}\label{eqconj}
f_{m_i,2}\le c\log m_i\quad \text{ for all }i\ge \kappa.
\end{align}
\end{conj}
The motivation for this conjecture is as follows. First of all, we already know that $(1-\epsilon)\log m_i<f_{m_i,1}<(1+\epsilon)\log m_i$, for all large $i$ almost surely, regardless of the value of $\theta$. In the case where $\theta\equiv0$, the difference $f_{m_i,2}-f_{m_i,1}$ has the same distribution as $f_{m_i,1}$. Since altering $\theta$ does not have any effect on the order of $f_{m_i}$, we suspect it does not have any substantial effect on $f_{m_i,2}-f_{m_i,1}$ either. Later, we show that the number of ``jumps'' (continuous stretches of fire, please see the proof of part~(b) of Lemma~\ref{lem_mlam} for the exact definition) in the fire reaching $m_i$ for the first time is of order $\log i$, which is quite low compared to the order of~$m_i$ that exceeds $\exp\{e^{\alpha i}\}$, as we show later. This indicates that the path (using the graphic representation language) of the first fire reaching $m_i$ is not ``very bumpy'', and thus it might behave more like a flat path (instantaneous burn). Hence, while we do not claim $f_{m_i,2}-f_{m_i,1}$ equals $(1+o(1)\log m_i$, we have a strong reason to believe that it is {\em at least} of order $\log m_i$.

\begin{thm}\label{Theorem 6}
Suppose that $\Delta=0$ and $\theta>0$ is constant.
Assume that Conjecture~\ref{conj1} is true.    
Then  $f_{n,1} = \mathcal{O}(\log n)$, i.e.\ $\exists$ $M>0$ such that $f_{n,1}\leq M\cdot \log n$ for large $n$.
\end{thm}

We start with the auxiliary statement.
\begin{lemma}\label{lem_mlam}
Suppose that $\Delta=0$ and $\theta>0$ is constant.
\begin{itemize}
    \item[(a)] Let $r>0$. Then a.s.\ $m_i\ge i^r$ for all but finitely many $i$.
    \item[(b)] Assume that Conjecture~\ref{conj1} is true. Let $\epsilon>0$ and $c$ be the constant from~\eqref{eqconj}.  Then a.s.\  $m_{i+1} \le  {m_i}^{c+\epsilon}$ for all but finitely many~$i$.
\end{itemize}
\end{lemma}
\begin{proof}
(a) The statement is trivial for $r\le 1$, since $m_i\ge i$, so from now on assume that $r>1$. Let $0<\epsilon<\frac1r$. Suppose that the fire has reached $m_i$ at time $t:=f_{m_i,1}$; w.l.o.g.\ we can assume that $t$ is sufficiently large.

\setlength{\unitlength}{1cm}
\begin{figure}
    \centering  
\begin{picture}(8,5)
  \thinlines 
  \put(0,0.4){\vector(1,0){8}}  
  \put(2,0.4){\line(0,1){4}}    
  \put(4,0.43){\line(0,1){4}}    
  \put(6,0.4){\line(0,1){4}}  
\put(0.7,1.3){$t'$}
\put(0.0,2.3){$t'+\theta$}
\put(0.7,1.8){$t$}
\put(0.0,2.9){$t+\theta$}
\put(0.0,3.9){$t+2\theta$}
\multiput(1, 2)(0.2, 0){15}{\line(1,0){0.1}} 
\multiput(1, 3)(0.2, 0){25}{\line(1,0){0.1}} 
\multiput(1.2, 4)(0.2, 0){24}{\line(1,0){0.1}} 
\multiput(1, 1.4)(0.2, 0){5}{\line(1,0){0.1}} 
\multiput(1, 2.4)(0.2, 0){5}{\line(1,0){0.1}} 
\thicklines 
\put(2.02,1.4){\line(0,1){1}}    
\put(1.98,1.4){\line(0,1){1}}    
\put(4.02,2){\line(0,1){1}}    
\put(3.98,2){\line(0,1){1}}    
\put(5.4,0){$m_i+1$}
\put(3.8,0){$m_i$}
\put(6,3.5){\circle*{0.15}}
\put(4,2){\circle*{0.15}}
\end{picture}
    \caption{how the second fire reaching $m_i$ can spread very far.}
    \label{fig3}
\end{figure}
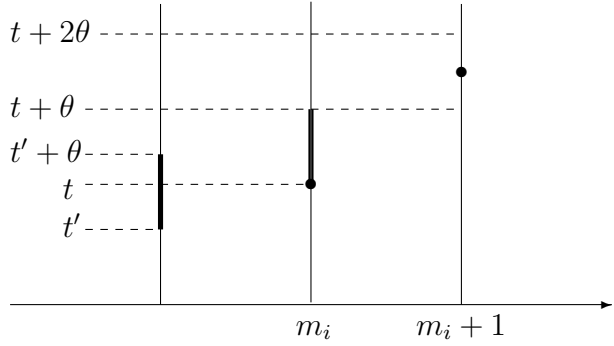

Since the fire has successfully spread from $m_i-1$ to $m_i$, this fire had reached $m_i-1$ at time~$t'\in( t-\theta,t]$, and thus with probability at least $e^{-2\theta}$ there will be no new tree at $m_i-1$ during the time interval $[t'+\theta,t+2\theta]$ (as $t'+\theta$ is the time when the tree at $m_i-1$ stops burning). Under this event, the {\em second fire} to reach $m_i$ cannot happen before time $t+2\theta$, which, in turn, provides the site $m_i+1$ with at least $\theta$ units of time to grow a tree (there was none during the time $[t,t+\theta]$ since we know that $m_i$ was the local maximum by that time); please see Figure~\ref{fig3}.
Hence, with probability at least
\begin{align}\label{eq3theta}
e^{-2\theta}\cdot \left(1-e^{-\theta}\right)    
\end{align}
$f_{m_i+1,1}\in[f_{m_i,2}, f_{m_i,2}+\theta)$, i.e., the second fire to reach $m_i$ will also reach $m_i+1$. However, by that time, there can already be the whole stretch of trees at sites $m_i+2,m_i+3,\dots$; in fact, the number $L_i$ of consecutive sites which do have a tree by this time has a geometric distribution with parameter $p=e^{-f_{m_i+1,1}}<e^{-t}$. Let $\tilde \G^*$ be the sigma-algebra containing all the events which happened by time $f_{m_i+1,1}$ to the left of $m_i+1$. Then
\begin{align}\label{eqLip}
\P\left(L_i\ge \frac 1p\mid \G^*\right)=\left(1-p\right)^{\lceil 1/p\rceil}=e^{-1}+o(1)    
\end{align}
where $o(1)\to 0$ as $p\to 0$ (i.e., $t\to \infty$). Let
\begin{align}\label{eqB}
B_i:=\left\{m_{i+1}\ge m_i+m_i^{1-\epsilon}\right\}
=\{L_i\geq m_i^{1-\epsilon}\}.
\end{align}
If event $F_{m_i}$ from Lemma~\ref{Lemma 3} occurs, then $t=f_{m_i,1}\ge (1-\epsilon)\log m_i$ yielding $1/p> m_i^{1-\epsilon}$. Taking into account the bound of probability~\eqref{eq3theta} and~\eqref{eqLip}, we obtain that 
\begin{align}\label{eqmi}
\P(B_i\mid  \G^*)\ge \left[ 
    e^{-2\theta-1}\cdot \left(1-e^{-\theta}\right)+o(1)\right] \times 1_{F_{m_i}}\ge c'\times 1_{F_{m_i}}
\end{align}
for some $c'>0$.

Let $\mu_i=m_i^\epsilon$. From Taylor expansion, we obtain that on $B_i$ we have 
$$
\mu_{i+1}-\mu_i\ge \epsilon+O(m_i^{-\epsilon})\ge \epsilon/2
$$
for sufficiently large $i$; also, trivially, $\mu_{i+1}\ge \mu_i$.
Consequently, we can create a sequence of i.i.d.\ Bernoulli$(c')$ random variables $\xi_i$, $i=1,2,\dots$, such that $\mu_{i+1}-\mu_i$ is stochastically larger than~$\frac{\epsilon}2\, \xi_i\, 1_{F_{m_i}}$. By the strong law, $\frac1n\sum_{i=1}^n\xi_i\to c'$ a.s., and at the same time
$$
0\le \frac1n\sum_{i=1}^n\xi_i-\frac1n\sum_{i=1}^n\xi_i\, 1_{F_{m_i}}\le \frac{|\{i:\ F_{m_i}\text{ did not occur}\}|}n
$$
By Lemma~\ref{Lemma 3}, the quantity in the numerator is a.s.\ finite and does not depend on $n$, hence
$$
\lim_{n\to\infty} \frac1n\sum_{i=1}^n\xi_i\, 1_{F_{m_i}}=c'\quad\text{a.s.}
$$
This, in turn, leads to
$$
\liminf_{i\to\infty} \frac{\mu_i}i\ge \frac{c'\,\epsilon}2\qquad \text{a.s.}
$$
yielding 
$$
m_i\ge  \left(\frac{c' \epsilon}3 i\right)^{1/\epsilon}\ge i^r\qquad\text{a.s.}
$$
for all sufficiently large $i$, since $1/\epsilon>r$.
${}$
\\[5mm]
(b) Consider the first fire that reaches beyond $m_i$ (the one which occurred at time $f_{m_i+1,1}$). This fire will extend further to the right in several so-called ``jumps'', rigorously defined in the next two paragraphs, until it eventually stops somewhere. The number of these jumps will be denoted by $N_i\in\{1,2,\dots\}$.

Let $L_{i,1}$ be the length of the maximum stretch of sites to the right of $m_i$ burned at the same time as $m_i+1$, so that $f_{m_i+y,1}=f_{m_i+1,1}$ for $y=2,3,\dots,L_{i,1}$. After the fire reaches site $m_i+L_{i,1}$, and $m_i+L_{i,1}+1$ is empty at that time (since we know that the fire did not immediately go further) there are two possibilities: either with probability $e^{-\theta}$ within the next~$\theta$ units of time no tree appears at the site $m_i+L_{i,1}+1$, or such a tree does appear. 

In the first case, we set $N_i=1$ and stop. In the latter case, we say that the fire has {\bf a jump} at position $m_i+L_{i,1}+1$, meaning that the fire, whilst burning both $m_i+L_{i,1}$  and $m_i+L_{i,1}+1$,  reached the latter site with a delay (up to $\theta$; $0<f_{m_i+L_{i,1}+1,1}-f_{m_i+L_{i,1},1}<\theta$). In case of such a jump, the fire will immediately spread to the additional~$L_{i,2}$ sites to the right of $m_i+L_{i,1}$, and thus $f_{m_i+L_{i,1}+y,1}=f_{m_i+L_{i,1}+1,1}$ for $y=2,3,\dots,L_{i,2}$, but there is no tree at the site $m_i+L_{i,1}+L_{i,2}+1$ at this time.

Similarly, we can define $L_{i,3},L_{i,4},\dots,L_{i,N_i}$, where~$N_i$ is the number of continuous stretches of sites that were burnt at the same time (before the fire eventually stopped completely);  $N_i$~is the number of such jumps. As a result, $N_i$ is a geometric random variable with rate~$e^{-\theta}$, and~the site $\left(m_i+\sum_{j=1}^k L_{i,j}\right)+1$ is burnt slightly later than the site $m_i+\sum_{j=1}^k L_{i,j}$ for each $k=1,2,\dots,N_i-1$, with the delay of up to~$\theta$. 

From the construction,
\begin{align}\label{eqxi}
m_{i+1}-m_i=\sum_{j=1}^{N_i} L_{i,j}.
\end{align}
Note also that $N_i$ is typically not very large, formally,  for some sufficiently large $c_1>0$
\begin{align}\label{eq:Ci}
\P(C_i^c)\le \left(1-e^{-\theta}\right)^{\lfloor c_1\log i\rfloor}\le \frac1{i^2}
\quad \text{ where }C_i:=\{N_i\le   c_1\log i\}
\end{align}
which implies by the Borel-Cantelli lemma that a.s.\ $N_i\le c_1\log i$ for all but finitely many $i$s.

Our goal now is to find an upper bound on how far the fire, that passed beyond $m_i$ for the first time, can travel (i.e., the first fire to burn $m_i+1$). Conjecture~\ref{conj1} postulates that the time to reach $m_i$ for {\em the second time} is at most $c\log m_i$ for all $i\ge \kappa$, and for these values of $i$ each $L_{i,j}$  can be coupled with $\xi_{i,j}$, $j=1,2,\dots$, where $\xi_{i,j}$ are independent geometric random variables with rate $e^{-(\theta(j-1)+c\log m_i) }=m_i^{-c} e^{-\theta (j-1)}$, respectively, such that $L_{i,j}\le \xi_{i,j}$. This coupling results from the fact that under this conjecture, the time to reach $m_i$ for the second time  $f_{m_i,2}\le c\log m_i$ and then after $j-1$ jumps the time is at most $\Delta:=f_{m_i,2}+\theta(j-1) \le c\log m_i+\theta(j-1)$, and the probability of not having a single tree by this time is $e^{-\Delta}$.

 Together with~\eqref{eqxi} this yields
\begin{align}\label{eqxii}
\left(m_{i+1}- m_i\right) 1_{i\ge\kappa} \preccurlyeq \sum_{j=1}^{N_i} \xi_{i,j}
\end{align}
where $X \preccurlyeq Y$ means that $Y$ stochastically dominates $X$.
Let 
$$
\mathcal{B}_i:=\left\{\sum_{j=1}^{N_i}\xi_{i,j}\ge m_i^{c+\epsilon/2}\times c_1\log i\right\}\supseteq 
\left\{\sum_{j=1}^{N_i}\xi_{i,j}\ge m_i^{c+\epsilon}-m_i\right\}
$$
for large $i$ (the inclusion follows from part (a)).
Then, using~\eqref{eq:Ci}, for $i\ge \kappa$
\begin{align}\label{eq:B}
\begin{split}
\P\left(\mathcal{B}_i\mid \F_{f_{m_i+1,1}}\right)&\le
\P\left(\mathcal{B}_i\mid N_i\le c_1 \log i\mid \F_{f_{m_i+1,1}}\right)
+
\P\left(N_i>c_1 \log i\mid \F_{f_{m_i+1,1}}\right)
\\ & \le 
\P\left( \xi_{i,j}\ge m_i^{c+\epsilon/2} \text{ for some }j=1,2,\dots,\lfloor c_1\log i\rfloor\right)
+\frac1{i^2}
\\ & \le 
c_1\log i \times \max_{j\le c_1\log i}\P\left( \xi_{i,j}\ge m_i^{c+\epsilon/2} \right)
+\frac1{i^2}
\end{split}
\end{align}
Since
\begin{align*}
\P\left(\xi_{i,j}> m_i^{c+\epsilon/2}\right)
\le \left(1-\frac1{m_i^c e^{\theta (j-1)}}\right)^{m_i^{c+\epsilon/2}}
\le \exp\left(-e^{-\theta(j-1)}\, m_i^{\epsilon/2}\right)
\le \exp\left(-\frac{m_i^{\epsilon/2}}{i^{c_1\theta}}\right)
\end{align*}
as $j\le c_1\log i$, and by part (a) of the lemma $m_i$ grows faster than $i^r$ for any $r>0$, the RHS of~\eqref{eq:B} is summable over $i$, and thus the events $\left\{\sum_{j=1}^{N_i}\xi_{i,j}\ge m_i^{c+\epsilon}-m_i\right\}$ happen finitely often by the Borel-Cantelli lemma. As a result, from~\eqref{eqxii}
$$
m_{i+1}\le  m_i^{c+\epsilon}
$$ 
for all large $i$ a.s.
\end{proof}

\begin{proof}[Proof of Theorem~\ref{Theorem 6}]
Fix some $\epsilon>0$. Lemma~\ref{lem_mlam} implies that $\log m_{i+1}\le (1+\epsilon) \log m_i$ for all large~$i$, a.s. On the other hand, by Lemma~\ref{Lemma 2}, $f_{m_{i+1},1}\le (1+\epsilon) \log m_{i+1}$.  Hence, for all $n \in \{m_i,\dots,m_{i+1}\}$ we have 
$$
f_{n,1}\leq f_{m_{i+1},1}\le   (1+\epsilon)^2\log{m_i}\le  (1+\epsilon)^2 \log{n}.
$$
Together with Lemma~\ref{Lemma 3}, this implies $f_{n,1} =\mathcal{O}(\log n)$.
\end{proof}
 
The next statement strengthens part (a) of Lemma~\ref{lem_mlam} ({\em without} using Conjecture~\ref{conj1}).
\begin{thm}
There exists an $\alpha>0$ such that $m_i\ge \exp(\exp(\alpha i))$ for all large $i$ a.s.
\end{thm}
\begin{proof}
Fix some $\epsilon>0$, and denote $\P_*(\cdot)=\P\left(\cdot\mid \F_{f_{m_{i},1}}\right)$, the conditional probability measure based on the information available until the first time $m_i$ was burnt; thus $\F_{f_{m_{i},1}}$ is a {\em stopped sigma-algebra}.

First, we want to estimate the time difference between the second and the first fires which burnt the next maximum site $m_{i+1}$. Define
$$
A_i=\{f_{m_{i+1},2}-f_{m_{i+1},1}\ge 2\epsilon\log m_i\}
$$
and recall the definitions of $F_n$ from Lemma~\ref{Lemma 3}, $B_i$ from~\eqref{eqB} and $C_i$ from~\eqref{eq:Ci}; the event $C_i$ is independent of $\F_t$, while the event $F_{m_i}$ is $\F_t$-measurable.

On $F_{m_i}$, we have $f_{m_{i+1},1}>f_{m_i,1}\ge(1-\epsilon)\log m_i$. Let us now estimate $\P_*\left(A_i^c\mid  F_{m_i}\right)$. The fire which reached $m_{i+1}$ for the first time is also the one which reached sites $x\in (m_{i},m_{i+1}]$ for the first time. Recall that $N_i$ denotes the number of jumps of this fire on these positions; from the proof of part (b) of Lemma~\ref{lem_mlam} it follows that with probability $\ge 1-1/i^2$ the event~$C_i=\{N_{i}\le \lfloor c\log i\rfloor\}\subseteq  \{e^{\theta N_{i}}\le i^{c\theta}\}$ occurs (we are putting an upper bound on the number of jumps in a fire and this fact did not require Conjecture~\ref{conj1}).

For  $f_{m_{i+1},2}-f_{m_{i+1},1}$  to be smaller than $2\epsilon \log m_i$, at least one tree must appear during the time interval $[f_{m_{i+1},1}-\theta N_i,f_{m_{i+1},1}+2\epsilon \log m_i]$ at each location  $x\in [m_i,m_{i+1})$   (the term $(-\theta N_i)$ corresponds to the fact that there are up to~$N_i$ jumps among these positions), and the probability of this event is at most $\left(1-\frac{1}{m_i^{2\epsilon}\, e^{N_i\theta}}\right)$ for each of these sites, independently of the others.  Hence
\begin{align*}
\P_*(A_i^c\mid F_{m_i})&\le \P_*(\text{tree appears at all $x$ for } m_i<x\le m_{i+1}\mid F_{m_i})
\\
&\le
 \P_*(\text{tree appears at all $x$ for $m_i<x\le m_{i+1}$}\mid B_i,F_{m_i})+\P_*(B_i^c\mid F_{m_i})
\\
&\le
 \P_*(\text{tree appears at all $x$ for $m_i<x\le m_{i+1}$}\mid B_i,C_i,F_{m_i})+\P(C_i^c\mid F_{m_i})+\P_*(B_i^c\mid F_{m_i})
\\ &
 \le
\left(1-\frac{1}{m_i^{2\epsilon}\, e^{\theta c\log i}}\right)^{m_i^{1-\epsilon}}+\frac1{i^2}
+(1-c')
= \left(1-\frac{1}{m_i^{2\epsilon}\, i^{c\theta}}\right)^{m_i^{1-\epsilon}}
+\frac1{i^2}+1-c'
\\ &
\le 
 \exp\left\{-\frac{m_i^{1-3\epsilon}}{ i^{c\theta}}\right\}
 +1-2 c'' \qquad\text{for large }i
\end{align*}
for some $c''>0$, since on $B_i$ we have $m_{i+1}-m_i\ge m_i^{1-\epsilon}$, and using~\eqref{eqmi}.

The longer the time interval between the first and second fire at $m_{i+1}$ (as implied by the event $A_i$), the greater the probability that it will burn at least $m_i^{1+\epsilon/2}$ additional sites. Hence, if both $F_{m_i}$ and $A_i$ occur,  the probability that {\em the second fire to arrive at $m_{i+1}$} will also stop before immediately burning additional $m_i^{1+\epsilon/2}$ sites equals 
\begin{align}\label{eq:m:2m}
&\P_*(f_{m_{i+1}+k,1}> f_{m_{i+1},2}\text{ for some }k\in [1,m_i^{1+\epsilon/2}]\mid F_{m_i},A_i)
\nonumber\\
&\le 
\P_*(f_{m_{i+1}+1,1}> f_{m_{i+1},2}\mid F_{m_i},A_i)
+\sum_{k=2}^{m_i^{1+\epsilon/2}}\P_*(f_{m_{i+1}+k,1}> f_{m_{i+1},2}\mid F_{m_i},A_i)\nonumber
\\
&\qquad \le
\frac{e^{-\theta}}{m_i^{2\epsilon}}+
\frac{m_i^{1+\epsilon/2}}{m_i^{1+\epsilon}}
\le \frac{2}{m_i^{\epsilon/2}}
\end{align}
(when $m_i$ is large). The first term in the above inequality corresponds to the probability that a tree appeared at position $m_{i+1}+1$ during the time interval 
$$
[f_{m_{i+1},1}+\theta, f_{m_{i+1},2}]\supseteq [f_{m_{i+1},1}+\theta, f_{m_{i+1},1}+2\epsilon \log m_i],
$$
while the second term corresponds to the event that {\em at least one} of the locations $m_{i+1}+2,m_{i+1}+3,\dots,m_{i+1}+m_i^{1+\epsilon/2}$ remains vacant, and we know that the probability of such a site to remain vacant is $\exp\{-(1+\epsilon)\log m_i\}=1/m_i^{1+\epsilon}$. Since~$F_{m_i}$ is $\P_*$-measurable, using~\eqref{eq:m:2m} we have
\begin{align*}
\P_*(m_{i+2}- m_i\ge m_i^{1+\epsilon/2})&
\ge 
\P_*(m_{i+2}- m_i\ge m_i^{1+\epsilon/2}, A_i, F_{m_i})
\\ &
=
\P_*(m_{i+2}- m_i\ge m_i^{1+\epsilon/2}\mid A_i, F_{m_i})
\, \P_*( A_i\mid F_{m_i}) \, \P_*( F_{m_i})
\\ &
\ge \left(1-\frac2{m_i^{\epsilon/2}}\right)
\cdot\left(2c''-
\exp\left\{-\frac{m_i^{1-3\epsilon}}{2\, i^{c\theta}}\right\}
 \right)1_{F_{m_i}}\ge c''\cdot  1_{F_{m_i}} \cdot 1_{m_i>i^r}
\end{align*}
for some fixed non-random $r>1$ and all large $i$. 
Now the inequality $m_{i+2}- m_i\ge m_i^{1+\epsilon/2}$ implies
\begin{align}\label{eqcomplim}
\log m_{i+2}\ge \left(1+\frac{\epsilon}2\right)\log m_i\quad\Longrightarrow\quad
\log\log m_{i+2}-\log\log m_i\ge \log  \left(1+\frac{\epsilon}2\right)=:\epsilon'>0.  
\end{align}
Setting $i=2j$, $j=1,2,\dots$, we conclude that we can construct an i.i.d.\ sequence of Bernoulli($c''$) random variables $\xi_j$ such that
$$
\log\log m_{2(j+1)}-\log\log  {m_{2j}}\ge \epsilon' \xi_j\cdot 1_{F_{m_{2j}}} \cdot 1_{m_{2j}>(2j)^r}.
$$
By Lemma~\ref{lem_mlam} $m_i>i^r$ for any fixed $r$ and all sufficiently large $i$; we also know that the events~$F_{m_i}^c$ occur only finitely often a.s.\ yielding
\begin{align*}
 \liminf_{k\to\infty} \frac{\log\log m_{2k}}k
 &=\liminf_{k\to\infty} \frac{\sum_{j=1}^{k-1}\left(\log\log m_{2j+2}-\log\log m_{2j}\right)}k
 \\ &
 \ge \liminf_{k\to\infty} \frac{\sum_{j=1}^{k-1}\epsilon'\xi_j \cdot 1_{F_{m_{2j}}} \cdot 1_{m_{2j}>(2j)^r}}k
  =\epsilon'\lim_{k\to\infty} \frac{\sum_{j=1}^{k-1}\xi_j}k=\epsilon'c''>0\qquad \text{a.s.}
  \end{align*}
(please compare this with the proof of part (a) of Lemma~\ref{lem_mlam}) from which it follows that for some $\alpha>0$ we have $m_i\ge \exp\left\{e^{\alpha i}\right\}$ for all sufficiently large $i$ a.s.
\end{proof}
\begin{remark}
Inequality \eqref{eqcomplim}  shows  that for some $\epsilon''>0$ we have $m_{i+1}\ge m_i^{1+\epsilon''}$   for all large~$i$ a.s., thus complementing part~(b) of Lemma~\ref{lem_mlam} (conditional on Conjecture~\ref{conj1}).
\end{remark}


\end{document}